\documentclass[12pt]{article}

\usepackage{amsmath,amsfonts,amssymb,amsthm}
\usepackage{enumitem}
\usepackage{fullpage}
\usepackage{bbm}  

\newtheorem{theorem}{Theorem}
\newtheorem{lemma}[theorem]{Lemma}
\newtheorem{corollary}[theorem]{Corollary}
\newtheorem{proposition}[theorem]{Proposition}
\newtheorem{observation}[theorem]{Observation}

\DeclareMathOperator{\Cayley}{Cay}

\newcommand{\NN}{\ensuremath{\mathbb N}}
\newcommand{\QQ}{\ensuremath{\mathbb Q}}
\newcommand{\ZZ}{\ensuremath{\mathbb Z}}
\newcommand{\B}{\ensuremath{\mathcal B}}
\newcommand{\C}{\ensuremath{\mathcal C}}

\renewcommand{\H}{\ensuremath{\mathcal H}}
\newcommand\bbone{\ensuremath{\mathbbm 1}}

\begin{document}

\title{Integral Cayley multigraphs over Abelian and Hamiltonian groups}

\author{
   Matt DeVos\thanks{Supported in part by an NSERC Discovery Grant (Canada) and a Sloan Fellowship.}\\[1mm]
   {Department of Mathematics}\\{Simon Fraser University}\\{Burnaby, B.C. V5A 1S6}
\and
   Roi Krakovski\thanks{Supported in part by postdoctoral support at the Simon Fraser University.}\\[1mm]
   {Department of Mathematics}\\{Simon Fraser University}\\{Burnaby, B.C. V5A 1S6}
\and
   Bojan Mohar\thanks{Supported in part by an NSERC Discovery Grant (Canada), by the Canada Research Chair program, and by the Research Grant P1--0297 of ARRS (Slovenia).}~\thanks{On leave from: IMFM \& FMF, Department of Mathematics, University of Ljubljana, Ljubljana, Slovenia.}\\[1mm]
   {Department of Mathematics}\\{Simon Fraser University}\\{Burnaby, B.C. V5A 1S6}
\and
   Azhvan Sheikh Ahmady\\[1mm]
   {Department of Mathematics}\\{Simon Fraser University}\\{Burnaby, B.C. V5A 1S6}
   }


\maketitle

\begin{abstract}
It is shown that a Cayley multigraph over a group $G$ with generating multiset $S$ is integral
(i.e.,\ all of its eigenvalues are integers) if $S$ lies in the integral cone over the boolean algebra generated by the normal subgroups of $G$. The converse holds in the case when $G$ is abelian. This in particular gives an alternative, character theoretic proof of a theorem of Bridges and Mena (1982).
We extend this result to provide a necessary and sufficient condition for a Cayley multigraph over a Hamiltonian group to be integral, in terms of character sums and the structure of the generating set.
\end{abstract}

\section{Introduction}

A graph $X$ is said to be \emph{integral\/} if all eigenvalues of the adjacency matrix
of $X$ are integers.  This property was first defined by Harary and Schwenk \cite{HS} who
suggested the problem of determining which graphs satisfy it.  Although this problem
appears far too complicated in general, there is a broad literature on it with numerous
interesting successes.

Throughout we shall assume that $G$ is a finite group written multiplicatively.
For a subset $S \subseteq G$ which is inverse-closed (i.e., $S=S^{-1}$), we define the
\emph{Cayley graph} of $G$ over $S$, denoted $\Cayley(G,S)$, to be the graph with vertex set $G$ and $x,y \in G$ adjacent if $xy^{-1} \in S$.
In what follows, we shall consider the multigraph version. Recall that a \emph{multiset\/} is a set $S$ together with multiplicity function $\mu_S: S\to \NN$, where $\mu_S(x)$ is a positive integer for every $x\in S$ (counting ``how many times $x$ occurs in the multiset''). We set $\mu_S(x)=0$ for $x\notin S$. A multiset $S$ of group elements is \emph{inverse-closed} if $\mu_S(s)=\mu_S(s^{-1})$ for every $s\in S$. If $S$ is an inverse-closed multiset of elements of a group $G$, the \emph{Cayley multigraph\/} $\Cayley(G,S)$ is defined as above except that it is a multigraph and the number of edges joining $x,y$ in $\Cayley(G,S)$ is equal to $\mu_S(xy^{-1})$.

Any collection $\mathcal S$ of subsets of $G$ naturally generates a boolean algebra of sets by taking all sets of elements of $G$ that are obtained from the sets in $\mathcal S$ by repeatedly taking unions, intersections and the symmetric difference.
We let $\B(G)$ denote the boolean algebra generated by the normal subgroups of $G$.
When extending this notion to multisets, we include multiset operations. Formally, we take all atoms of the Boolean algebra $\B(G)$ and take all multisets that can be expressed as non-negative integer combinations of these atoms. This defines the collection $\C(G)$ of multisets that is called the \emph{integral cone} over $\B(G)$.

The following theorem gives a complete characterization of which Cayley multigraphs over abelian groups are integral.

\begin{theorem}[Bridges, Mena \cite{BM}]
\label{thm:Bridges and Mena}
If\/ $G$ is an abelian group, then $\Cayley(G,S)$ is integral if and only if $S \in {\mathcal C}(G)$.
\end{theorem}

Although the result of Bridges and Mena was originally stated for simple Cayley graphs, these authors proved the multigraph version.  For each group element $g \in G$, let $A_g$ denote the permutation matrix (indexed by $G \times G$) associated with $g$ and for a set $S \subseteq G$ let $A_S = \sum_{s \in S} A_s$.
Bridges and Mena proved that for an abelian group $G$, a complex linear combination of the matrices $\{A_g : g \in G\}$ is a rational matrix with rational eigenvalues if and only if it is a rational combination of the matrices $\{A_Q : \mbox{$Q$ is an atom of $\B(G)$}\}$.

There has been some recent interest in finding a new proof of this result.  Notably, So \cite{So} found a new proof in
the special case when $G$ is a cyclic group, and Klotz and Sander \cite{KS} found a new proof of the ``if'' direction for all abelian groups.  Our initial
goal was to give a new proof of Theorem \ref{thm:Bridges and Mena}.  Our proof is based on characters, and is a fairly direct generalization of that given by So.
However, our approach generalizes to non-abelian groups and enables us to consider more general classes of groups.

First, we obtain a sufficient condition for integrality of Cayley multigraphs over arbitrary groups.

\begin{theorem}
\label{thm:sufficiency}
For every finite group $G$ and every $S\in {\mathcal C}(G)$, the Cayley multigraph $\Cayley(G,S)$ is integral.
\end{theorem}

In the course of proving Theorem \ref{thm:sufficiency}, we obtain a new proof of Theorem \ref{thm:Bridges and Mena} based on character theory.

It is worth mentioning that having used normal subgroups and not all subgroups in the definition of $\B(G)$ and ${\mathcal C}(G)$ is necessary for Theorem \ref{thm:sufficiency} to hold. There are examples (e.g.\ dihedral groups) where using a subgroup in the generating set does not yield integral Cayley graphs.

By Theorem \ref{thm:Bridges and Mena}, the converse statement to Theorem \ref{thm:sufficiency} holds in abelian groups. In the second part of the paper we investigate to what extent the converse would hold in some other groups. As a natural candidate, we have decided to consider \emph{dedekind groups}, i.e. groups whose every subgroup is normal. Every abelian group is dedekind; non-abelian dedekind groups are also called \emph{hamiltonian groups}, and they have a simple characterization that is due to Baer, cf.\ \cite[Theorem 12.5.4]{Hall}.

\begin{theorem}
\label{thm:hamiltonian groups}
A finite group is hamiltonian if and only if it can be written as a direct product $Q_8\times A$,
where $Q_8$ is the group of quaternions and $A$ is an abelian group without elements of order $4$.
\end{theorem}

We provide sufficient and necessary conditions for integrality of the spectra of Cayley multigraphs over such groups (Theorem \ref{thm:hamiltonian case}).
By using this characterization, we show that integrality of Cayley graphs over hamiltonian groups is easy to decide in certain special cases, while it leads to challenging combinatorial problems in some other special cases. See Sections \ref{sect:Hamiltonian} and \ref{sect:special cases} for more details.

\section{$B$-integrality}

If ${\mathcal G}$ is a collection of graphs on the common vertex set $V$ and $B$ is an orthogonal basis of ${\mathbb C}^V$, then we say that ${\mathcal G}$ is $B$-\emph{integral\/} if for every $X$ in  ${\mathcal G}$, $B$ is a set of eigenvectors for $X$ and all eigenvalues of $X$ are integral. Equivalently, if $A(X)$ denotes the adjacency matrix of $X$, then $A(X)B = B\Lambda$, where $\Lambda$ is a diagonal matrix with integer entries (and $B$ is viewed as a matrix whose columns are the vectors from $B$). If $X$ and $Y$ are (simple) graphs on the same vertex set $V$, then we denote by $X\cup Y$ the simple graph on $V$ in which vertices $u,v$ are adjacent if and only if they are adjacent in $X$ or in $Y$ (or in both). For any family of graphs ${\mathcal G}$ on a common vertex set, we let $U({\mathcal G})$ be the closure of ${\mathcal G}$ under the operation $\cup$. We begin with an easy lemma.

\begin{lemma} \label{lem:5}
{\rm (a)} If\/ $X$ is $B$-integral and\/ $\bbone \in B$, then $\overline{X}$ is $B$-integral.

{\rm (b)} If\/ $X$ and $X \cap Y$ are $B$-integral, then $X \cap \overline{Y}$ is $B$-integral.

{\rm (c)} If\/ ${\mathcal G}$ is an intersection-closed family of $B$-integral graphs, then $U({\mathcal G})$ is $B$-integral.
\end{lemma}

\begin{proof}
To prove (b), observe that $A(X\cap \overline{Y}) = A(X) - A(X\cap Y)$, thus
$$
   A(X\cap \overline{Y})B = A(X)B - A(X\cap Y)B = B\Lambda_1 - B\Lambda_2 = B(\Lambda_1-\Lambda_2).
$$
Since $\Lambda_1$ and $\Lambda_2$ are integral, so is their difference, hence $X \cap \overline{Y}$ is $B$-integral.

If $\bbone \in B$, then the complete graph $K_n$ is $B$-integral since $A(K_n)$ has eigenvalue $n-1$ with eigenvector $\bbone$ and eigenvalue $-1$ with eigenspace $\bbone^{\perp}$. Thus (a) follows from (b) by taking $K_n$ and $X$ playing the role of $X$ and $Y$, respectively.

By observing that $A(X\cup Y) = A(X) + A(Y) - A(X\cap Y)$, a proof similar to the above proof of (b) shows that (c) holds.
\end{proof}

For any set of graphs ${\mathcal G}$ on a common vertex set $V$, we let $ \B({\mathcal G})$ denote the set of all graphs
on $V$ which may be expressed using members of ${\mathcal G}$ and the operations $\cap$, $\cup$, and complement.

\begin{lemma} \label{lem:6}
Let ${\mathcal G}$ be an intersection-closed family of $B$-integral graphs and assume that $\bbone \in B$.  Then $\B({\mathcal G})$ is $B$-integral.
\end{lemma}

\begin{proof}
Set ${\mathcal G}_0 = {\mathcal G}$ and for every $k \in {\mathbb N}$ recursively define
\[ {\mathcal G}_{k+1} = \{ X_1 \cap \cdots \cap X_n : \mbox{either $X_i \in {\mathcal G}_k$ or $\overline{X_i} \in {\mathcal G}_k$ for every $1 \le i \le n$} \}. \]
It is immediate that each ${\mathcal G}_k$ is intersection-closed and it follows from De Morgan's law that $\B({\mathcal G}) = \cup_{k=0}^{\infty} {\mathcal G}_k$.  To complete the proof, we shall show, by induction on $k$, that every graph in ${\mathcal G}_k$ is $B$-integral.
As a base, we observe that this holds for
$k=0$ by assumption.  For the inductive step, let $X$ be a graph in $G_{k+1}$ and suppose that
$X = X_1 \cap \cdots \cap X_{\ell} \cap \overline{Y_1} \cap \cdots \cap \overline{Y_m}$ where $X_1, \ldots, X_{\ell},Y_1, \ldots, Y_m \in {\mathcal G}_k$.  Then we have
\[ X = \left( \cap_{i=1}^{\ell} X_i \right) \cap \overline{ \left( \cup_{j=1}^m Y_j \right) }. \]
Since ${\mathcal G}_k$ is intersection-closed and $X'=X_1 \cap \cdots \cap X_{\ell} \in {\mathcal G}_k$ and $X'\cap (\cup_{j=1}^m Y_j) = ( \cup_{j=1}^m (X'\cap Y_j) \in U({\mathcal G}_k)$,
it follows from Lemma \ref{lem:5} that $X$ is $B$-integral, as desired.
\end{proof}

\begin{lemma}
\label{lem:commute}
Let $S$ and $T$ be inverse-closed multisets of a group $G$. If\/ $g\,T=Tg$ (equality holding as multisets) for every $g\in G$, then the adjacency matrices of Cayley multigraphs $\Cayley(G,S)$ and\/ $\Cayley(G,T)$ commute.
\end{lemma}

\begin{proof}
Let $A_S$ and $A_T$ be the adjacency matrices of both Cayley multigraphs, and let $g,h\in G$. Since $S$ and $T$ are inverse-closed, we have
\begin{align*}
   (A_SA_T)_{g,h} &= \sum_{x\in G} \mu_S(gx^{-1})\mu_T(xh^{-1}) =
                    \sum_{x\in G} \mu_{Sg}(x)\mu_{Th}(x)\\[2mm]
      &= \sum_{x\in G} \mu_{Sg}(xg)\mu_{Th}(xg) =
         \sum_{x\in G} \mu_{S}(x)\mu_{Thg^{-1}}(x).
\end{align*}
Taking a similar expression for $A_TA_S$ and using the fact that $S$ and $T$ are inverse-closed, we derive
\begin{align*}
   (A_TA_S)_{g,h} &= \sum_{x\in G} \mu_{Tgh^{-1}}(x)\mu_{S}(x) =
                    \sum_{x\in G} \mu_{Tgh^{-1}}(x^{-1})\mu_{S}(x^{-1}) \\[2mm]
      &= \sum_{x\in G} \mu_{hg^{-1}T}(x)\mu_{S}(x).
\end{align*}
To obtain equality for every $g$ and $h$, it suffices to see that $\mu_{Thg^{-1}}(x) = \mu_{hg^{-1}T}(x)$ for every $g,h,x\in G$; equivalently,
$\mu_{Tk}(x) = \mu_{kT}(x)$ for every $k,x\in G$. But this is precisely our assumption that $Tk=k\,T$.
\end{proof}

Lemma \ref{lem:commute} implies that the adjacency matrices of all Cayley multigraphs $\Cayley(G,T)$, where $T$ is any normal subgroup of $G$ (or any other union of conjugacy classes), commute. Therefore, they have a common set of eigenvectors.

\bigskip

\noindent{\it Proof of Theorem \ref{thm:sufficiency}. }
Let ${\mathcal G} = \{ \Cayley(G,H) : H \vartriangleleft G \}$.  By Lemma \ref{lem:commute}, the adjacency matrices of all graphs in ${\mathcal G}$ commute and hence they have a common orthogonal set $B$ of eigenvectors. For every $X \in {\mathcal G}$ we have that $X$ has $B$ as a
basis of eigenvectors, and $X$ is a disjoint union of cliques (with loops at every vertex), so $X$ is $B$-integral.  It now follows from Lemma \ref{lem:6} that
$\B({\mathcal G}) = \{ \Cayley(G,S) : S \in  \B(G) \}$ is $B$-integral.
Since the adjacency matrix $A_T$ of each multigraph $\Cayley(G,T)$, $T\in {\mathcal C}(G)$, is an integral linear combination of adjacency matrices of $\Cayley(G,S)$, $S\in \B(G)$, also $A_T$ is $B$-integral. This completes the proof.
\hfill$\Box$

\section{The boolean algebra of normal subgroups}

In this short section, we shall describe the atoms of $\B(G)$.
For an element $x\in G$, let $N(x)$ be the intersection of all normal subgroups containing $x$. For $x,y\in G$, define $x\sim y$ if $N(x)=N(y)$.
This equivalent to the condition that $y\in N(x)$ and $x\in N(y)$.
Clearly, $\sim$ is an equivalence relation in $G$.

\begin{proposition}
\label{prop:atoms}
The equivalence classes of\/ $\sim$ are precisely the atoms of the boolean algebra $\B(G)$, and therefore $S\in \B(G)$ if and only if\/ $S$ is the union of such equivalence classes, and a multiset is in $\C(G)$ if and only if it is a multiset generated by the $\sim$-classes.
\end{proposition}

\begin{proof}
Suppose that $x\sim y$. Then $x\in N(y)$ and $y\in N(x)$. Thus, every normal subgroup of $G$ either contains both or neither of these two elements. This property is preserved by taking unions, intersections and complements. Thus, every $S\in \B(G)$ is the union of equivalence classes of $\sim$.

Suppose now that $x\not\sim y$. If $x\in N(y)$ and $y\in N(x)$, then $N(x)=N(y)$. Thus, we may assume that $x\notin N(y)$. Since $N(y)$ is in $\B(G)$, the atom of $\B(G)$ containing $y$ does not contain $x$. This shows that each atom is an equivalence class of $\sim$.
\end{proof}

\section{Abelian groups}

Throughout this section, $G$ will always be a (finite) abelian group.
Let $G^*$ denote the dual group of $G$, consisting of all complex characters, i.e.,
homomorphisms from $G$ to (the multiplicative group of) ${\mathbb C}$.  It is well
known that $G^*$ is a group under pointwise multiplication, and that $G^* \cong G$.
We define $F$ to be the matrix indexed by $G^* \times G$ and given by the rule that for $\alpha \in G^*$ and $x \in G$ we have
$F_{\alpha,x} = \alpha(x)$.  Note that each row of $F$ is a character.  Furthermore, it follows from the fact that the characters form an orthogonal
basis that $F F^* = nI$, where $F^*$ is the conjugate transpose of $F$ and $n=|G|$.
Finally, observe that if $r$ is the exponent of $G$, then every element of $F$ is an $r^{\mathrm{th}}$ root of unity.

In the remainder, for any vector $v \in {\mathbb C}^I$ (where $I$ is a non-empty index set) and any $n \in {\mathbb Z}$, we let $v^n$ denote the vector in ${\mathbb C}^I$ given by coordinate-wise exponentiation, i.e., $(v^n)_i = (v_i)^n$ for each $i\in I$.

\begin{observation}
\label{obs:1}
Let $x,y \in G$ and let $F_x, F_y$ denote the column vectors of $F$ indexed by $x$ and $y$, respectively.  If $x \sim y$, then there exist
integers $j,k \in {\mathbb Z}$ so that $(F_x)^j = F_y$ and $(F_y)^k = F_x$.
\end{observation}

\begin{proof}
Since $x \sim y$, we may choose $j,k \in {\mathbb Z}$ so that $x^j = y$ and $y^k = x$.  Now, for any character $\alpha \in G^*$ we have
$\alpha(y) = \alpha(x^j) = (\alpha(x))^j$ and it follows that $F_y = (F_x)^j$.  A similar argument shows that $F_x = (F_y)^k$.
\end{proof}

\begin{lemma}
\label{lemma7}
Let\/ $v \in {\mathbb Q}^G$. If\/ $Fv \in {\mathbb Q}^{G^*}$, then for every $x,y \in G$ with $x \sim y$, we have $v_x = v_y$.
\end{lemma}

\begin{proof}
Let $F_x$ and $F_y$ denote the column vectors of $F$ indexed by $x$ and $y$ and let $\ell$ $(m)$ be the smallest integer so that
every term of $F_x$ $(F_y)$ is a $\ell^{\,\mathrm{th}}$ $(m^{\mathrm{th}})$ root of unity.  It follows from Observation \ref{obs:1} that $\ell = m$.  Now, fix
a primitive $\ell^{\,\mathrm{th}}$ root of unity $\omega$ and express each entry of $F_x$ and $F_y$ in the form $\omega^i$ for some $i \in \{0,1,\ldots,\ell-1\}$.
Using this interpretation, and recalling that $u := Fv$ is rational, we obtain an expression for the inner product of $F_x$ and $u$ as
\[ F_x \cdot u = (F^*u)_x = \sum_{i=0}^{\ell-1} a_i \, \omega^i \]
where each $a_i \in {\mathbb Q}$. Note that $F_x \cdot u = (F^*u)_x = nv_x$. Now, let $P(z) \in {\mathbb C}[z]$ denote the polynomial $P(z) = \sum_{i=0}^{\ell-1} a_i z^i - nv_x$.
Observe that $P(\omega)=0$.  Next, choose $j \in \{0,1,\ldots,\ell-1\}$ so that $F_y = (F_x)^j$.  Note that $\gcd(j,l)=1$.  We may assume $x \neq y$, as otherwise there is nothing to prove. It follows that $j \ge 2$, so $\ell \ge 3$.  The polynomial $P$ has rational coefficients and has $\omega$ as a root.  It follows from this and the fact that the  polynomial
\[\Phi_{\ell}(z) = \prod_{i \in \{1..\ell\} : \gcd(i,\ell) =1} (z - \omega^i)\]
is irreducible over ${\mathbb Q}$, that $\omega^j$ is also a root of $P$. But then we have
\[ 0 = P(\omega^j) = \sum_{i=0}^{\ell-1} a_i \omega^{ij} - nv_x = F_y \cdot u - nv_x \]
which implies that $v_y = \tfrac{1}{n} F_y \cdot u = v_x$ as desired.
\end{proof}

We are ready to give a proof of Theorem \ref{thm:Bridges and Mena}. Observe that sufficiency part follows from Theorem \ref{thm:sufficiency}.

\bigskip

\noindent{\it Proof of Theorem \ref{thm:Bridges and Mena}, necessity. }
It is well known that for every $S \subseteq G$ the Cayley graph $\Cayley(G,S)$ has each character $\alpha \in G^*$ as an eigenvector
with eigenvalue $\alpha(S) = \sum_{g \in S} \alpha(g)$.  Alternately, if we view $\alpha$ as a vector and let $\bbone_S \in {\mathbb C}^G$ denote the characteristic vector of $S$, this eigenvalue may be written as $\alpha \cdot \bbone_S$.
Suppose that $\Cayley(G,S)$ is integral. Then we have $\alpha \cdot \bbone_S \in {\mathbb Q}$ for every $\alpha \in G^*$.
Equivalently, $F\/\bbone_S$ is rational-valued.  But then, it follows from the previous lemma that whenever $x,y \in G$ satisfy
$x \sim y$, we have $(\bbone_S)_x = (\bbone_S)_y$.  By Proposition \ref{prop:atoms}, this implies that $S \in \B(G)$, as desired.
\hfill$\Box$

\section{Hamiltonian groups}
\label{sect:Hamiltonian}

Let $\H$ be the family of groups of the form $Q_8 \times A$ where $A$ is a finite abelian group and $Q_8$ is the quaternion group represented as follows:
$$Q_8=\langle -1,i,j,k \mid (-1)^2=1, i^2=j^2=k^2=ijk=-1\rangle.$$

Let us recall that a finite group $G$ is {\em hamiltonian} if it is non-abelian and every subgroup of $G$ is normal. By a well-known result of Baer (Theorem \ref{thm:hamiltonian groups}), every hamiltonian group is in $\H$.

In this section, we obtain a necessary and sufficient condition for a multigraph $\Cayley(G,S)$ to be integral, where $G \in \H$ and $S \subseteq G$ is an inverse-closed multiset of elements of $G$.

We require some definitions and notation from representation theory. For a more detailed account the reader is referred to \cite{FH,JL}.

In what follows, let $G$ be a finite group. We denote by $\mathbb{C}G$ its group algebra. That is, $\mathbb{C}G$ is the vector space over $\mathbb{C}$ with basis $G$ and multiplication defined by extending the group multiplication linearly.
Identifying $\sum_{g \in G} a_g g$ with the function $g \mapsto a_g$, we can view the vector space $\mathbb{C}G$ as the space of all $\mathbb{C}$-valued functions on $G$.
We sometimes identify a subset or a multiset $M$ of $G$ with the element
$\sum_{g\in G} \mu_M(g) g$ of the group algebra $\mathbb{C}G$.

Let $V$ be an $n$-dimensional vector space over $\mathbb{C}$. A \emph{complex representation} (or simply a \emph{representation}) of $G$ on $V$ is a group homomorphism $\rho: G \rightarrow GL(V)$, where $GL(V)$ denotes the group of invertible endomorphism of $V$. The \emph{degree} of representation is the dimension of $V$.
Two representations $\rho_1$ and $\rho_2$ of $G$ on $V_1$ and $V_2$, respectively, are \emph{equivalent} (written $\rho_1 \cong \rho_2$) if there is a linear isomorphism $T: V_1 \to V_2$ such that for every $g\in G$ we have $T\rho_1(g) = \rho_2(g)T$.

If $\rho$ is a representation of $G$ then the \emph{character} $\chi_{\rho}$ induced by $\rho$, is the linear functional $\chi_{\rho}: \mathbb{C}G\rightarrow \mathbb{C}$ defined by
$$\chi_{\rho}(g)= {\rm tr}(\rho(g)),\quad g \in G,$$
and extended by linearity to $\mathbb{C}G$.  (The trace ${\rm tr} (\alpha)$ of a linear map $\alpha$ is the trace of any matrix representing $\alpha$ according to some basis.)
The \emph{degree} of the character $\chi_{\rho}$ is the degree of  $\rho$, and is equal to $\chi_{\rho}(1)$. A character of degree one is called a \emph{linear character}. The character $\chi$ which assigns $1$ to every element of a group $G$ is called the \emph{principal character} of $G$ and is denoted by ${\mathbbm 1}_G$.

The $|G|$-dimensional representation $\rho_{reg}:G \rightarrow GL(\mathbb{C}G)$ defined by $\rho_{reg}(g)(x)=gx$ ($g \in G, x \in \mathbb{C}G)$) is called the
\emph{left-regular representation}. Choosing $G$ as a basis for $\mathbb{C}G$, we see that for every $g \in G$, $[\rho_{reg}(g)]_G$ (the matrix representing the linear map $\rho_{reg}(g)$ according to the basis $G$) is the $(|G| \times |G|)$-matrix, indexed with the elements of $G$, such that for every $k,h \in G$

\[([\rho_{reg}(g)]_G)_{h,k} = \left\{
\begin{array}{l l}
  1 & \quad \mbox{if $h=gk$ }\\
  0 & \quad \mbox{otherwise.}
\end{array} \right. \]

This gives a natural link to Cayley graphs since the adjacency matrix of a Cayley multigraph $\Cayley(G,S)$ can be written as
\begin{equation}
    A = \sum_{s\in S}\ \mu_S(s)[\rho_{reg}(s)]_G.
\label{eq:adjacency regular representations}
\end{equation}

Let $\rho: G \rightarrow GL(V)$ be a representation. A subspace $W$ of $V$ is said to be $\rho$\emph{-invariant}, if $\rho(g)w \in W$ for every $g \in G$ and $w \in W$.  If $W$ is a $\rho$-invariant subspace of $V$, then the restriction of $\rho$ to $W$, that is $\rho|_W:G \rightarrow GL(W)$, is a representation of $G$ on $W$. If $V$ has no $\rho$-invariant subspaces, then  $\rho$ is said to be an \emph{irreducible representation} for $G$ and the corresponding character $\chi_{\rho}$ an \emph{irreducible character} for $G$.

For a group $G$, we  denote by ${\rm IRR}(G)$ and ${\rm Irr}(G)$  a complete set of  non-equivalent irreducible representations of $G$ and the complete set of  non-equivalent irreducible characters of $G$, respectively. Note that ${\rm IRR}(G)$ is not necessarily unique, but ${\rm Irr}(G)$ is unique.
If $G$ is abelian, then every irreducible representation $\rho\in {\rm IRR}(G)$ is 1-dimensional and thus it can be identified with its character $\chi_\rho\in {\rm Irr}(G)$.

The following result was observed by many authors (see for example~\cite{DS}).

\begin{lemma}
\label{lemma:eig}
Let\/ $G$ be a group and let $S\subseteq G$ be a multiset of elements of $G$. Let\/ ${\rm IRR}(G)=\{\rho_1,\dots,\rho_k\}$. For\/ $t=1,\dots,k$, let $d_t$ be the degree of $\rho_t$, and let $\Lambda_t$ be the multiset of eigenvalues of the matrix $\sum_{g \in S}\mu_S(g)\rho_t(g)$. Then the following holds:
\begin{enumerate}
\item[\rm (1)] The set of eigenvalues of\/ $\Cayley(G,S)$ equals
$\cup_{t=1}^k \Lambda_t$.
\item[\rm (2)] If the eigenvalue $\lambda$ occurs with multiplicity $m_t(\lambda)$ in
$\sum_{g \in S}\mu_S(g)\rho_t(g)$ $(1\le t\le k)$, then the multiplicity of $\lambda$ in $\Cayley(G,S)$ is $\sum_{t=1}^k d_t m_t(\lambda)$.
\end{enumerate}
\end{lemma}

The table below is the character table of $Q_8$.

\begin{center}
\begin{tabular}{c|ccccc}
$g \in Q_8$ & $1$ & $-1$ & $i$& $j$& $k$\\
\hline
$cl(g)$ & $\{1\}$& $\{-1\}$ & $\{i, -i\}$ & $\{j, -j\}$ & $\{k, -k\}$ \\
\hline
${\mathbbm 1}_{Q_8}$& 1&1&1&1&1\\
$\lambda_i$ & 1&1&1&-1&-1\\
$\lambda_j$& 1&1&-1&1&-1\\
$\lambda_k$ & 1&1&-1&-1&1\\
$\varepsilon$ & 2 & -2& 0&0&0\\
\end{tabular}
\end{center}
where $\varepsilon$ is the character afforded by the representation $\rho_{\varepsilon}$ defined below:
\newcommand\ii{{\tt i}}
$$\rho_{\varepsilon}(1)= I, \hspace{1cm}
\rho_{\varepsilon}(i) = \left( \begin{array}{cc}\ii&0\\0&-\ii\\ \end{array} \right), \hspace{1cm}
\rho_{\varepsilon}(j)= \left(\begin{array}{cc}0&1\\ -1 &0 \\ \end{array} \right), \hspace{1cm}
\rho_{\varepsilon}(k)= \left( \begin{array}{cc} 0&\ii \\ \ii&0 \\ \end{array} \right)$$
where the value $\ii$ appearing in the matrices is the complex imaginary unit $\sqrt{-1}$. Also note that $\rho_{\varepsilon}(-g)= -\rho_{\varepsilon}(g)$ for every $g \in Q_8$ and that ${\rm IRR}(Q_8) = \{ {\mathbbm 1}_{Q_8}, \lambda_i, \lambda_j, \lambda_k,\rho_{\varepsilon}\}$.

Let $G= Q_8 \times A$, where $A$ is an abelian group, and let $S \subseteq G$ be an inverse-closed multiset of elements of $G$. For every $q \in Q_8$, let $B_q$ be the multiset
\begin{equation}
  B_q=\{a\in A\mid (q,a)\in S\}
\label{eq:Bq}
\end{equation}
in which the multiplicity of $a\in B_q$ is equal to the multiplicity of $(q,a)$ in $S$.

Since $S$ is inverse-closed, we have $B_1=B_1^{-1}$, $B_{-1}=B_{-1}^{-1}$, and $B_{-q}=B_q^{-1}$ for every $q \in Q_8 \setminus  \{1, -1\}$. In particular, this implies that $\lambda(B_{-q}) = \overline{ \lambda(B_q)}$, for every $\lambda \in {\rm Irr}(A)$.
For every multiset $D$ of elements of $A$, we define
$$
   \widehat{\lambda}(D) = \lambda(D) - \lambda(D^{-1}) =
   \sum_{g\in D}\mu_D(g)(\lambda(g)-\lambda(g^{-1}))=
   \sum_{g\in D}\mu_D(g)(\lambda(g)-\overline{\lambda(g)}).
$$
In particular, for every $q \in Q_8$, $\widehat{\lambda}(B_q) = \lambda(B_q)-\lambda(B_{-q})$.
The following is the main result of this section.

\begin{theorem}
\label{thm:hamiltonian case}
Let $G= Q_8 \times A$, where $A$ is an abelian group, and let $S$ be an inverse-closed multiset of elements of $G$.  Then $\Cayley(G,S)$ is integral if and only if the following holds:
 \begin{enumerate}
\item [\rm{(i)}] $B_1, B_{-1} \in \C(A)$.
\item [\rm{(ii)}] The multiset union $B_q \cup B_{-q}\in \C(A)$, for every $q \in Q_8 \setminus \{-1,1\}$.
\item [\rm{(iii)}]$ \widehat{\lambda}(B_i)^2 + \widehat{\lambda}(B_j)^2 + \widehat{\lambda}(B_k)^2$ is a negative square of an integer, for every $\lambda \in {\rm Irr}(A)$.
\end{enumerate}
\end{theorem}

\begin{proof}
Lemma~\ref{lemma:eig} shows that $\Cayley(G,S)$ is integral if and only if the matrix $\sum_{s\in S} \mu_S(s)\phi(s)$ is integral for every $\phi \in {\rm IRR}(G)$.
Since $\phi$ is an irreducible representation of the direct product $Q_8\times A$, it can be written in the form $\phi = \rho \times \lambda$ for some $\rho \in {\rm IRR}(Q_8)$ and $\lambda \in
{\rm Irr}(A)$ (where we identify ${\rm Irr}(A)$ and ${\rm IRR}(A)$ since all irreducible representations of $A$ are 1-dimensional).
In other words, $\phi(q,a)=\lambda(a)\rho(q)$ for every $(q,a)\in Q_8\times A$.
Consequently, $\Cayley(G,S)$ is integral if and only if the matrices
$$
   A^{(\rho,\lambda)} = \sum_{(q,a)\in S} \mu_S((q,a))(\rho \times \lambda)(q,a) = \sum_{(q,a)\in S}\mu_S((q,a))\lambda(a) \rho(q)
$$
are integral for every $\rho \in {\rm IRR}(Q_8)$ and every $\lambda \in {\rm Irr}(A)$. By definition of $B_q$ we can write the matrix $A^{(\rho,\lambda)}$ in the following form:
\begin{equation}
\label{myeq}
  A^{(\rho,\lambda)} = \sum_{(q,a)\in S} \mu_S((q,a))\lambda(a)\rho(q) = \sum_{q \in Q_8} \lambda(B_q)\rho(q).
\end{equation}

Integrality of the matrix in (\ref{myeq}) (with $\rho=\rho_\varepsilon$ and $\lambda \in {\rm Irr}(A)$ arbitrary) together with the fact that the trace of a matrix is equal to the sum of its eigenvalues implies that  $${\rm tr}\Bigl(\sum_{q \in Q_8} \lambda(B_q)\rho_\varepsilon(q)\Bigr)= \sum_{q \in Q_8} \lambda(B_q)\varepsilon(q) = 2  (\lambda(B_1) - \lambda(B_{-1})) \in \mathbb{Z}.$$

It follows that
\begin{equation}
\label{eq:1holds}
 \lambda(B_1) - \lambda(B_{-1}) \in \mathbb{Q}.
\end{equation}

Let $\rho\in \{{\mathbbm 1}_{Q_8},\lambda_i,\lambda_j,\lambda_k\}$ be a degree-one representation of $Q_8$ and let $\lambda\in {\rm Irr}(A)$. Define $\lambda^+(B_q) = \lambda(B_q)+\lambda(B_{-q})$. Observe that $\rho(q) = \rho(-q)$ for every $q\in Q_8$. Therefore, integrality of the matrices $A^{(\rho,\lambda)}$ in (\ref{myeq}) implies by the same argument as above that
$$
  \rho(1)\lambda^+(B_1) + \rho(i)\lambda^+(B_i) + \rho(j)\lambda^+(B_j) + \rho(k)\lambda^+(B_k) \in \ZZ.
$$
This yields the following four conditions (one for each $\rho\in \{1,\lambda_i,\lambda_j,\lambda_k\}$):
\begin{eqnarray}
  \lambda^+(B_1) + \lambda^+(B_i) + \lambda^+(B_j) + \lambda^+(B_k) &\in& \ZZ \nonumber\\
  \lambda^+(B_1) + \lambda^+(B_i) - \lambda^+(B_j) - \lambda^+(B_k) &\in& \ZZ \nonumber\\
  \lambda^+(B_1) - \lambda^+(B_i) + \lambda^+(B_j) - \lambda^+(B_k) &\in& \ZZ
  \label{eq:system 4x4}\\
  \lambda^+(B_1) - \lambda^+(B_i) - \lambda^+(B_j) + \lambda^+(B_k) &\in& \ZZ \nonumber
\end{eqnarray}
Since the matrix of coefficients of the linear system (\ref{eq:system 4x4}) is invertible, this implies that $\lambda^+(B_q)\in\QQ$ for every $q\in Q_8$. In particular, since $\lambda^+(B_1) = \lambda(B_1)+\lambda(B_{-1})\in\QQ$, we conclude by using (\ref{eq:1holds}) that $\lambda(B_1)\in\QQ$ and $\lambda(B_{-1})\in\QQ$, while for $q\in Q_8\setminus\{1,-1\}$, we have $\lambda(B_q) + \lambda(B_{-q}) \in \QQ$.

Rationality of $\lambda(X)$ for every $\lambda\in {\rm Irr(A)}$ has been discussed in the proof of Lemma \ref{lemma7}, where it was proved that this is equivalent to the condition that $X\in \C(A)$. Therefore, the conclusions stated in the previous paragraph imply (i) and (ii).

Conversely, notice that by Theorem \ref{thm:Bridges and Mena}, (i) and (ii) imply integrality of the matrices in (\ref{myeq}), where $\rho$ is any degree-one representation of $Q_8$ and $\lambda \in {\rm Irr}(A)$.

For (iii), we consider the degree-two representation $\rho_{\varepsilon}$. As observed above, $\rho_{\varepsilon}(-q)= -\rho_{\varepsilon}(q)$ for every $q\in Q_8$, and hence
$$ \sum_{q \in Q_8} \lambda(B_q)\rho_{\varepsilon}(q) = \widehat{\lambda}(B_1)I +  \widehat{\lambda}(B_i)\left( \begin{array}{cc}\ii&0\\0&-\ii\\ \end{array} \right) + \widehat{\lambda}(B_j)\left(\begin{array}{cc}0&1\\ -1 &0 \\ \end{array} \right)+\widehat{\lambda}(B_k)\left( \begin{array}{cc} 0&\ii \\ \ii&0 \\ \end{array} \right).$$
As mentioned above, (\ref{eq:1holds}) implies that $\widehat{\lambda}(B_1) \in \mathbb{Z}$.
Therefore, $\sum_{q \in Q_8} \lambda(B_q)\rho_{\varepsilon}(q)$ is integral if and only if the matrix
$$
   M = \left( \begin{array}{cc} \ii\widehat{\lambda}(B_i)&\widehat{\lambda}(B_j) + \ii \widehat{\lambda}(B_k)\\ -\widehat{\lambda}(B_j) + \ii \widehat{\lambda}(B_k)& -\ii\widehat{\lambda}(B_i) \\ \end{array} \right).
$$
is integral. By considering the characteristic polynomial of $M$, it is easy to see that
$M$ is integral if and only if $\widehat{\lambda}(B_i)^2 + \widehat{\lambda}(B_j)^2 + \widehat{\lambda}(B_k)^2$ is the negative square of an integer. Hence, integrality of $\Cayley(G,S)$ implies (iii), and conversely, (iii) implies integrality of the matrices $A^{(\rho,\lambda)}$. This completes the proof.
\end{proof}

\section{Some special cases}
\label{sect:special cases}

In this section we consider some special cases of hamiltonian groups by applying Theorem \ref{thm:hamiltonian case}. This result gives a simple characterization in some cases, and leads to interesting combinatorial problems in some other cases.

\subsection{Simple Cayley graphs of $Q_8\times C_p$, where $p\ne 3$}

As the first special case of using Theorem~\ref{thm:hamiltonian case} we consider hamiltonian groups $G=Q_8\times C_p$, where $p \neq 3$ is a prime and $C_p$ is the cyclic group of order $p$. In analogy with the abelian case, we obtain the following complete characterization for integrality of \emph{simple} Cayley graphs over this group. The multigraph version is different and is treated in a separate section.

\begin{theorem}
\label{thm:Q8Cp}
Let\/ $p\neq 3$ be a prime and let $S\subseteq Q_8\times C_p$ be an inverse-closed subset of $ Q_8\times C_p$. The Cayley graph $\Cayley(Q_8\times C_p,S)$ is integral if and only if $S\in \B(Q_8\times C_p)$.
\end{theorem}

This result is a direct consequence of the following:

\begin{theorem}
\label{th:app1}
Let\/ $G=Q_8\times C_p$, for a prime $p \neq 3$. Let\/ $S \subseteq G$ be an inverse-closed subset of $G$, and let\/ $B_q$ $(q\in Q_8)$ be defined as in {\rm(\ref{eq:Bq})}. Then $\Cayley(G,S)$ is integral if and only if the following conditions hold:
\begin{enumerate}
\item [{\rm (P1)}] $B_1, B_{-1} \in \B(C_p).$
\item[{\rm (P2)}] For every $q \in Q_8 \setminus \{1,-1\}$, $B_q=B_{-q} \in \B(C_p)$.
\end{enumerate}
\end{theorem}

\begin{proof}
By Theorem~\ref{thm:hamiltonian case}, it suffices to show that (P2) holds if and only if conditions (ii) and (iii) in Theorem~\ref{thm:hamiltonian case} hold.
The ``only if'' part is trivial, since (P2) implies that $\widehat{\lambda}(B_i)=\widehat{\lambda}(B_j)=\widehat{\lambda}(B_k)=0$. For the ``if'' part suppose that conditions (ii) and (iii) of Theorem~\ref{thm:hamiltonian case} hold. By  condition (iii), for every $\lambda \in {\rm Irr}(C_p)$ there is an integer $\alpha_{\lambda}$ so that
\begin{equation}
\label{my:eq2}
\widehat{\lambda}(B_i)^2 + \widehat{\lambda}(B_j)^2 + \widehat{\lambda}(B_k)^2= -\alpha_{\lambda}^2.
\end{equation}

Let $e$ be the unit in $C_p$ and let $E_1=\{e\}$ and $E_2=C_p \setminus E_1$ be the two equivalence classes of $C_p$.
Let $q \in Q_8 \setminus \{1,-1\}$. Recall that since $S$ is inverse-closed we have $B_q^{-1} = B_{-q}$.

If $B_q \in \B(C_p)$ then $B_{-q}=B_q$.  This is true because $B_q^{-1} = B_{-q}$, and  the sets $E_1$ and $E_2$ are inverse-closed. Hence in this case $B_q = B_{-q}$ and  $\widehat{\lambda}(B_q) =0$. If $p=2$, then every subset of $C_p$ is in $\B(C_p)$, so (P2) holds in this case, and we may henceforth assume that $p\ge5$.

If $B_q \not\in \B(C_p)$, then by condition (ii) and the fact that $B_q^{-1} = B_{-q}$, we conclude that $E_2 \subseteq B_q \Delta B_{-q}$, thus the support of $B_q - B_{-q}$, viewed as an element of the group algebra $\mathbb{C}C_p$, contains $p-1$ distinct elements (that is, the whole class $E_2$), where an element and its inverse appear with opposite signs. In particular, the sum of coefficients of elements of $B_q - B_{-q}$ is 0.

Let us write $B'_q = B_q\setminus B_{-q}$, and observe that for every $q\in Q_8$, either $B'_q=\emptyset$ or $|B'_q| = \frac{1}{2}(p-1)$. Now, (\ref{my:eq2}) can be written as follows
\begin{align}
-\alpha_{\lambda}^2 &= \widehat{\lambda}(B_i)^2 + \widehat{\lambda}(B_j)^2 + \widehat{\lambda}(B_k)^2 \nonumber \\[1.5mm]
 &= \widehat{\lambda}(B'_i)^2 + \widehat{\lambda}(B'_j)^2 + \widehat{\lambda}(B'_k)^2 \nonumber \\[1.5mm]
 &= \lambda((B'_i - B'_{-i})^2) + \lambda((B'_j - B'_{-j})^2) + \lambda((B'_k - B'_{-k})^2) \nonumber\\[1.5mm]
 &= \lambda((B'_i - B'_{-i})^2 +(B'_j- B'_{-j})^2 +(B'_k - B'_{-k})^2) \nonumber\\
 &= \lambda\Bigl(-2(|B'_i| +|B'_j| +|B'_k|)e + \sum_{g \in E_2} a_g g\Bigr) \nonumber\\
\label{my:eq 1}
 &= -2(|B'_i| +|B'_j| +|B'_k|) +\lambda\Bigl(\sum_{g \in E_2} a_g g\Bigr)
\end{align}
where $a_g \in \mathbb{Z}$ for every $g\in E_2$. Since the sum of coefficients in $B_q-B_{-q}$ is zero, it follows that the sum of coefficients in $(B_q-B_{-q})^2$ is also zero. Thus, (\ref{my:eq 1}) implies that
\begin{equation}
\label{eq:added1}
  \sum_{g \in E_2} a_g =2(|B'_i| +|B'_j| +|B'_k|).
\end{equation}
By (\ref{my:eq 1}), $\lambda(\sum_{g \in E_2} a_gg) \in \mathbb{Q}$ for every $\lambda\in {\rm Irr}(C_p)$. It follows by Lemma \ref{lemma7}, that all coefficients $a_g$ are equal, and from (\ref{eq:added1}) we conclude that for every $g \in E_2$:
$$
   a_g = \frac{2(|B'_i| +|B'_j| +|B'_k|)}{p-1}.
$$
We also know that for each non-principal character $\lambda \in {\rm Irr}(C_p)$ we have $ \sum_{g \in C_p}\lambda(g) = 0$.
Thus, $\sum_{g \in E_2}\lambda(g) = -1$, and we can rewrite (\ref{my:eq 1}) as follows:
$$ -\alpha_{\lambda}^2 = -2(|B'_i| +|B'_j| +|B'_k|) -\frac{2(|B'_i| +|B'_j| +|B'_k|)}{p-1}.$$
This gives the following conclusion:
\begin{equation}
\label{my:eq 2}
  \alpha_{\lambda}^2 = 2(|B'_i| +|B'_j| +|B'_k|)\, \frac{p}{p-1}.
\end{equation}
We know that for every $q \in Q_8$, $|B'_q|$ is either $0$ or $\frac{1}{2}(p-1)$. By (\ref{my:eq 2}),
$p-1$ divides $2(|B'_i| +|B'_j| +|B'_k|)$. Let $\beta$ denote the number of elements $q \in \{i,j,k\}$ such that $|B'_q|=\frac{p-1}{2}$. Then we conclude from (\ref{my:eq 2}) that $\alpha_{\lambda}^2 =\beta p$.
Since $0 \le \beta \leq 3$ and $p\ge5$, this is possible only when $\alpha=0$. However, in that case (P2) holds.
\end{proof}

\subsection{$Q_8\times C_3$}

The conclusion of Theorem \ref{th:app1} does not hold for $p=3$. An example is provided in the next observation.

\begin{observation}
\label{obs:C3}
Let $G=Q_8 \times C_3$, and $S=\{(i,1),(-i,2),(j,1),(-j,2),(k,1),(-k,2)\}$. Then $\Cayley(G,S)$ is integral but\/ $S\not\in \B(G).$
\end{observation}

To see this, we verify conditions (i)--(iii) of Theorem \ref{thm:hamiltonian case}. Conditions (i) and (ii) are obvious; (iii) is left to the reader.

This graph is indeed a very interesting vertex-transitive graph whose properties are discussed below. Let us remark at this point that the proof of the Theorem~\ref{th:app1} shows that the example in Observation \ref{obs:C3} is the only integral simple Cayley graph of $Q_8\times C_3$ (up to Cayley graph isomorphisms and up to choice of $B_1,B_{-1}\in \B(C_3)$) that fails to satisfy the conclusion of Theorem~\ref{th:app1}.

This graph has a natural tripartition according to the first coordinate, and the bipartite graphs obtained from it by removing one of these tripartite classes is the
M\"obius-Kantor graph.  The M\"obius-Kantor graph is the unique double-cover of the cube of girth $6$ and it sits naturally as a subgraph of the $4$-cube. The graph of the $24$-cell is also tripartite with classes of size $8$, and deleting any one yields a $4$-cube.

\subsection{Cayley multigraphs of $Q_8\times C_p$}

Theorem \ref{th:app1} does not hold for the multigraph case. In this section we shall provide infinitely many examples confirming this.
We let $C_p = \{ a^t  \mid  0 \leq  t < p \}$, the cyclic group of order $p$ generated by $a$. We consider the multisets $B_q$ ($q \in  Q_8 \setminus \{1,-1\}$)  defined as in (\ref{eq:Bq}), and we set $B_1=B_{-1}=\emptyset.$ In order to satisfy conditions (i)--(iii) of Theorem \ref{thm:hamiltonian case}, we need that $B_q \cup B_{-q}\in \C (C_p)$ and $ \widehat{\lambda}(B_i)^2 + \widehat{\lambda}(B_j)^2 + \widehat{\lambda}(B_k)^2$ is a negative square of an integer, for every $\lambda \in {\rm Irr}(C_p)$. As before, for every $q \in Q_8 \setminus \{1,-1\}$ we define $B'_q=B_q \setminus (B_q \cap B_{-q})$, where $B_q \cap B_{-q}$ is
the multiset in which the multiplicity of any $x\in C_p$ is equal to the minimum of multiplicities of $x$ in $B_q$ and in $B_{-q}$. Thus, in particular, $\widehat{\lambda}(B_q) = \widehat{\lambda}(B'_{q})$.
Note that $B'_q$ and $B'_{-q}$ are disjoint and the condition that generating multiset is inverse-closed is equivalent to the requirement that the multiplicity of $a^t$ ($0\leq t <p$) in $B_q'$ is the same as the multiplicity of $a^{-t}$ in $B'_{-q}$. The following is a well-known result from number theory.

\begin{lemma}
If $p$ is a prime number and $p \equiv 1\ ({\rm mod}\ 4)$, then the diophantine equation $x^2 +y^2 = pz^2$ has infinitely many solutions satisfying $gcd(x,y,z)=1$.
\end{lemma}

A solution of the diophantine equation $x^2 +y^2 = pz^2$ is \emph{primitive} if $gcd(x,y,z)=1$. Clearly, every integral multiple of $(x,y,z)$ is also a solution. The solution $(0,0,0)$ is called the \emph{trivial solution}.

\begin{lemma}
Let $(r,s,t)$ be a non-trivial solution for the diophantine equation $x^2 +y^2 = 5z^2$. Let $D_1=ra+sa^2$, $D_2=ra+sa^3$ and $D_3=0$, be elements of $\mathbb{C}C_5$, where $a$ is a generator of $C_5$. Then
$$ \widehat{\lambda}(D_1)^2 + \widehat{\lambda}(D_2)^2 + \widehat{\lambda}(D_3)^2 = -(5t)^2.$$
\end{lemma}

\begin{proof}
For any $x$ and $y$ in $\mathbb{R}$, we have the following equation in $\mathbb{C}C_5$:
$$ (x(a-a^4) + y(a^2-a^3))^2= -2(x^2+y^2)+(y^2-2xy)(a+a^4)+(x^2+2xy)(a^2+a^3).$$
Thus,
$$\widehat{\lambda}(D_1)^2 =-2(r^2+s^2)+(s^2-2rs)\lambda (a+a^4)+(r^2+2rs)\lambda (a^2+a^3), $$
$$\widehat{\lambda}(D_2)^2 =-2(r^2+s^2)+(r^2+2rs)\lambda (a+a^4)+(s^2-2rs)\lambda (a^2+a^3).$$
Clearly, $\widehat{\lambda}(D_3)^2 =0.$
We notice also that for each non-principal character $\lambda \in {\rm Irr}(C_5)$ we have $\sum_{i=1}^4 \lambda(a^i) = -1$. Therefore, $\widehat{\lambda}(D_1)^2 + \widehat{\lambda}(D_2)^2 + \widehat{\lambda}(D_3)^2= -5(r^2+s^2)=-(5t)^2$.
\end{proof}

\begin{corollary}
\label{cor:18}
There are infinitely many multisets $S$ (none of which is a multiple of another) such that $\Cayley(Q_8 \times C_5,S)$ is integral but $S\notin \C(G)$.
\end{corollary}

\begin{proof}
Let us start with a primitive solution $(m,n,\alpha)$ of the diophantine equation $x^2 +y^2 = 5z^2$. Since $(2m,2n, 2\alpha )$ is a solution of  the diophantine equation $x^2 +y^2 = 5z^2$, we can construct $D_1$, $D_2$ and $D_3$ as in the previous lemma, i.e., $D_1 = 2ma+2na^2$, $D_2=2na+ 2ma^3$ and $D_3=0$. Suppose without loss of generality that $n\leq m$.
Let us take $B_i= \{2ma, (m+n)a^2, (m-n)a^3 \}$, $B_j=\{(m+n)a, (m-n)a^4, 2ma^3\}$, $B_{-i}=B_i^{-1}$, $B_{-j}=B_j^{-1}$, and $B_k=B_{-k}=\emptyset$ (where the coefficients of $a^t$ in the the set notation denote multiplicities). Then $B_i + B_{-i} \in \C(C_5)$, $B_j + B_{-j} \in \C(C_5)$, and $B_k + B_{-k} \in \C(C_5)$.
We also have $B'_i = D_1$, $B'_j=D_2$ and $B'_k=D_3$.
From the previous lemma we get
$$
 \widehat{\lambda}(B_i)^2 + \widehat{\lambda}(B_j)^2 + \widehat{\lambda}(B_k)^2= \widehat{\lambda}(B'_i)^2 + \widehat{\lambda}(B'_j)^2 + \widehat{\lambda}(B'_k)^2=-(10\alpha )^2.
$$
So, clearly conditions (i)--(iii) of Theorem \ref{thm:hamiltonian case} are satisfied for the generating multiset $S$ arising from $B_i$, $B_j$ and $B_k$, but $B_i,B_j \notin \C(C_5)$. Thus $S \notin \C(Q_8 \times C_5)$, while according to Theorem \ref{thm:hamiltonian case} the Cayley graph
$\Cayley(Q_8 \times C_5,S)$ is integral.
\end{proof}

The case $p=7$ is similar. First, we observe that there are infinitely many primitive solutions for diophantine equation $x^2+y^2+z^2=7\alpha ^2$. If we assume $(m,n,l, \alpha)$ is one of these solutions, then we can define
$$B'_i= ma+na^2+la^3, \qquad B'_j= la+ma^2+na^3, \qquad B'_k= na+la^2+ma^3.$$
It is easy to see that condition (iii) of Theorem \ref{thm:hamiltonian case} holds. As in Corollary \ref{cor:18}, we can define $B_i$, $B_j$ and $B_k$ using correspondence with $B'_i$, $B'_j$ and $B'_k$ such that conditions (i)--(iii) of Theorem \ref{thm:hamiltonian case} are satisfied. The details are left to the reader.
This gives rise to integral Cayley multigraphs of $Q_8\times C_7$ whose generating multiset is not in the lattice ${\mathcal C}(Q_8\times A)$.

\subsection{Simple Cayley graphs of $Q_8\times C_p^d$}

As the last special case we consider the group $G=Q_8\times C_p^d$, where $p$ is a prime and $d\geq 2$. Here the abelian direct factor of $G$ is an elementary abelian $p$-group, thus every non-identity element has order $p$. If $[a]$ denotes the equivalence class containing $a$ with respect to the relation $\sim$ in $C_p^d$ and if $a\ne e$ (where $e$ is the identity element of $C_p^d$), then
$[a] = \{a^t \mid 1\leq t \leq p-1\}$.
We also know that $[e]=\{e\}$; we call this the trivial equivalence class. Since each non-identity element in $C_p^d$ has order $p$, each non-trivial class is of order $p-1$, and the number of
non-trivial classes is equal to $n_d=\frac{p^d-1}{p-1}$. Label these classes as $A_r$ for $1\leq r \leq n_d$. If $\lambda$ is a non-principal character of $C_p^d$, then $|{\rm Im}(\lambda)|=p$ and therefore $ker(\lambda)$ is a subgroup of order $p^{d-1}$.

Let us assume that $\Cayley(G,S)$ is integral. Then we derive in the same way as in the case of $Q_8\times C_p$ that there is an integer $\alpha_\lambda$ such that
\begin{equation}
\label{my:eq 3}
-\alpha_{\lambda}^2=\widehat{\lambda}(B'_i)^2 + \widehat{\lambda}(B'_j)^2 + \widehat{\lambda}(B'_k)^2= -2(|B'_i| +|B'_j| +|B'_k|) + \lambda \Bigl(\sum_{g \in C_p^d \setminus \{e\}} a_gg\Bigr).
\end{equation}
Since (\ref{my:eq 3}) holds for every $\lambda\in {\rm Irr}(C_p^d)$, we conclude by Lemma \ref{lemma7} that the coefficients $a_g$ are constant on each equivalence class $A_r$. Let $b_r$ be the common value for $a_g$, $g\in A_r$. Then
\begin{align}
  -\alpha_{\lambda}^2 &= -2(|B'_i| +|B'_j| +|B'_k|) + \lambda\Bigl(\sum_{r=1}^{n_d}\sum_{g \in A_r} b_rg \Bigr) \nonumber\\
  &= -2(|B'_i| +|B'_j| +|B'_k|) + \sum_{r=1}^{n_d} b_r\lambda(A_r). \label{eq:br}
\end{align}
Since each $A_r \cup \{e\}$ is a subgroup of order $p$, we have
\begin{equation}
  \lambda(A_r) = \sum_{g \in A_r} \lambda(g)= \left\{
  \begin{array}{rl}
   p-1, & \quad \mbox{ $A_r\subseteq ker(\lambda)$ }\\
   -1, & \quad \ \mbox{$A_r\nsubseteq ker(\lambda)$.}
  \end{array} \right. \label{eq:kernel}
\end{equation}
We also notice that for $q\in\{i,j,k\}$ the element $B_q-B_{-q}$ of the group algebra has the sum of the coefficients equal to zero. By using this fact in combination with (\ref{eq:br}) and (\ref{eq:kernel}) for the case when $\lambda$ is the principal character and noting that $\alpha_\lambda = 0$ in that case, we obtain the following analogue of (\ref{my:eq 2}):
\begin{equation}
\label{my:eq 4} 2(|B'_i| +|B'_j| +|B'_k|) = (p-1)\sum_{r=1}^{n_d}b_r.
\end{equation}
Using (\ref{my:eq 4}), we have for every non-principal character $\lambda$:
\begin{equation}
   \alpha_{\lambda}^2 = 2(|B'_i| +|B'_j| +|B'_k|) - \sum_{r=1}^{n_d}b_r\lambda(A_r)=\sum_{r=1}^{n_d}b_r(p-1 - \lambda(A_r)).\label{eq:analogue}
\end{equation}
The equality (\ref{eq:kernel}) shows that a non-zero contribution in the sum on the right side of (\ref{eq:analogue}) arises only when $A_r \not\subseteq ker(\lambda)$. Let $I_\lambda\subseteq \{1,\dots,n_d\}$ be the set of values $r$ for which $A_r \not\subseteq ker(\lambda)$. Then we have:
\begin{equation}
\label{eq:b_r}
   \alpha_{\lambda}^2 = \sum_{r\in I_\lambda} b_r(p-1-\lambda(A_r)) = p\sum_{r\in I_\lambda}b_r.
\end{equation}
There is a natural geometric setting for these equations.  View $C_p^d$ as a vector space over $C_p$ and consider the projective geometry $PG(d-1,p)$ consisting of all subspaces of $C_p^d$.  The points in our projective geometry are the 1-dimensional subspaces of $C_p^d$ which are in correspondence with $A_1, A_2, \ldots, A_{n_d}$, and we label the point associated with $A_i$ by $b_i$.  The kernels of the non-principal characters of $C_p^d$ correspond to the hyperplanes in our projective geometry
(i.e.\ subspaces of dimension $d-1$ of $C_p^d$).  So, equation \ref{eq:b_r} implies that the sum of the labels on the complement of every hyperplane must be a square divided by $p$.  Although this is a meaningful constraint, it is not difficult to find labelings of the points in a projective geometry which satisfy this property, so a more complicated analysis will be required to understand the integrality of such Cayley graphs.

\end{document}